\documentclass[11pt,leqno]{amsart}
\usepackage[margin=20truemm, includefoot]{geometry}
\usepackage[utf8]{inputenc}
\usepackage{amsmath,amssymb,mathrsfs,stmaryrd,amsaddr}
\usepackage{bm}
\usepackage{enumerate}
\usepackage{amsthm}
\usepackage{hyperref}
\usepackage{comment}
\usepackage[all]{xy}
\usepackage{ascmac}
\usepackage{fancybox}
\usepackage{graphicx}
\usepackage{tikz}
\usepackage{amsfonts}
\usepackage[bbgreekl]{mathbbol}
\usepackage{physics}
\usepackage{datetime}
\usepackage{shuffle}
\DeclareSymbolFontAlphabet{\mathbb}{AMSb}
\DeclareSymbolFontAlphabet{\mathbbl}{bbold}
\DeclareMathSymbol{\bbepsilon}{\mathord}{bbold}{"0F}
\xyoption{poly}
\newtheoremstyle{mystyle}% % Name
    {}%                      % Space above
    {}%                      % Space below
    {\normalfont}%           % Body font
    {}%                      % Indent amount
    {\bfseries}%             % Theorem head font
    {}%                      % Punctuation after theorem head
    { }%                     % Space after theorem head, ' ', or \newline
    {}%                      % Theorem head spec (can be left empty, meaning `normal')
\DeclareMathOperator{\Ker}{Ker}
\DeclareMathOperator{\End}{End}

\DeclareMathOperator{\polylog}{Li}
\DeclareMathOperator{\depth}{dp}

\newcommand{\duarrbinom}[2]{\left\lbrace\mqty{#1\\#2}\right\rbrace}

\newtheorem{mainthm}{Theorem}

\newtheorem{defi}{Definition}[section]
\newtheorem{prop}[defi]{Proposition}

\newtheorem{cor}[defi]{Corollary}
\newtheorem{theo}[defi]{Theorem}

\newtheorem{lem}[defi]{Lemma}

\newtheorem{notation}[defi]{Notation}

\theoremstyle{remark}
\newtheorem{exa}[defi]{Example}
\newtheorem{rem}[defi]{Remark}
\newtheorem*{rem*}{Remark}
\newtheorem*{exa*}{Example}

\makeatletter

\@addtoreset{equation}{section}
\makeatother
\def\N{{\mathbb N}}
\def\C{{\mathbb C}}
\def\Z{{\mathbb Z}}

\def\Q{{\mathbb Q}}

\title[Non-positive MPL and combinatorics of Magnus polynomials]{\bf Multiple polylogarithms at non-positive indices and combinatorics of Magnus polynomials}
\author{Kohei Kitamura}
\address[]{Department of Mathematics, The University of Osaka, Toyonaka, Osaka 560‐0043, Japan.}
\email{u347758b@ecs.osaka-u.ac.jp; kitamurakmath@gmail.com}
\date{}
\begin{document}
\begin{abstract}
    In this paper we investigate multiple polylogarithms with non-positive multi-indices (non-positive MPLs) from a combinatorial and algebraic viewpoint. By introducing a correspondence between non-positive multiple polylogarithms and Magnus polynomials in a free associative algebra, we obtain an explicit Magnus-type representation of products of mono-indexed non-positive MPLs. The main identity (Theorem \ref{main4}) expresses such a product as a single non-positive MPL indexed by a Magnus polynomial, which may be regarded as a Möbius inversion of the expansion formula due to Duchamp–Hoang Ngoc Minh–Ngo. Moreover, we study the effects of permuted indices and show that certain differences of Magnus polynomials belong to the kernel of the linear map $\polylog^-_\bullet$, leading to new functional equations among non-positive MPLs of the same weight and depth. These results clarify the combinatorial structure underlying non-positive MPLs and reveal a close connection with the Magnus expansion in non-commutative algebra.
\end{abstract}
\maketitle
\tableofcontents
\section{Introduction}
    The \textit{multiple polylogarithm function} (hereafter \textit{MPL}) of depth $r$ is defined on the open unit disc $D=\{z\in\mathbb{C}\mid|z|<1\}$ for any \textit{(multi-)index} $\mathbf{s}=(s_1,\ldots,s_r)\in\mathbb{Z}_{>0}^r$ by:
    \begin{equation}\label{def.polylog}
        \polylog_{\mathbf{s}}(z):=\sum_{n_1>\cdots>n_r>0}\frac{z^{n_1}}{n_1^{s_1}\cdots n_r^{s_r}}.
    \end{equation}
    It is well-known that this function admits the following iterated integral representation as an iterated integral on the domain $D \subset \mathbb{C}$ (see \cite{Wa02}). Let $\omega_0:=dt/t$ and $\omega_1:=dt/(1-t)$ be differential 1-forms. Then we have:
    \begin{equation}\label{eq.ite.int}
        \polylog_{\mathbf{s}}(z)=\int_0^z\omega_0^{s_1-1}\omega_1\cdots\omega_0^{s_r-1}\omega_1.
    \end{equation}
    From this, it follows that the function can be analytically continued to the simply connected domain $\Omega:=\mathbb{C}\setminus((-\infty, 0]\cup [1,+\infty))\subset\mathbb{C}$. These equations (\ref{def.polylog}), (\ref{eq.ite.int}) form a basis for various branches of research from arithmetic, combinatorial or algebro-geometrical viewpoints (cf. \cite{Go98}, \cite{F04}, \cite{IKZ06} etc.).\par
    These MPLs $\polylog_\mathbf{s}$ have been studied also for non-positive indices by many authors (cf. \cite{D17a}, \cite{D17b}, \cite{EMS17}, \cite{GZ17}, \cite{Bu23}). In this paper, we write
    \begin{equation}
        \polylog^-_{\mathbf{s}}(z):=\polylog_{(-s_1,\ldots,-s_r)}(z)=\sum_{0<n_r<\cdots<n_1}n_1^{s_1}\cdots n_r^{s_r}z^{n_1}\quad(z\in D),
    \end{equation}
    and focus on their combinatorial properties.
    Each $\polylog^-_\mathbf{s}$ is a $\Q$-rational function of $z$ vanishing at $z=0$ and having a pole at $z=1$. Indeed, if $\theta_0:=z\partial_z$ and $\lambda:=z/(1-z)$, then $\polylog^-_{\mathbf{s}}$ has another expression (cf: \cite[(8)(9)(10)]{D17a}):
    \begin{equation}\label{eq.another ex. of n.p.MPL}
        \polylog^-_{\mathbf{s}}(z)=\theta_0^{s_1}(\lambda\cdot\theta_0^{s_2}(\cdots\lambda\cdot\theta^{s_r}(\lambda))\cdots)\in\frac{z}{1-z}\cdot\Q\qty[z,\frac{1}{1-z}],
    \end{equation}
    where $\lambda\cdot-$ in this equation means multiplication by $\lambda$ in $\mathcal{O}(\Omega)$.\par
    To study combinatorial behaviors of $\polylog^-_\mathbf{s}$ (for short non-positive MPL), we associate them to linear sums of indices as follows (cf. \cite[p.171]{D17a}). Let $X:=\{x_0,x_1\},\,Y=\{y_0,y_1,\ldots\}$ and $\C\langle X\rangle,\,\C\langle Y\rangle$ be the free associative $\C$-algebra generated by $X,Y$ respectively. Then we define the $\C$-linear map:
    \begin{equation}
        \polylog^-_\bullet:\C\langle Y\rangle\longrightarrow\C[z,(1-z)^{-1}]
    \end{equation}
    by $w=y_{s_1}\cdots y_{s_r}\longmapsto\polylog^-_w=\theta_0^{s_1}(\lambda\cdot\theta_0^{s_2}(\cdots\lambda\cdot\theta^{s_r}(\lambda))\cdots).$\par
    The purpose of this paper is to present a new combinatorial perspective of $\polylog^-_\bullet$ in view of their connection with \textit{Magnus polynomials}. Define the Lie bracket $[u,v]:=uv-vu$ for $u,v\in \C\langle x_0,x_1\rangle$, and set
    \begin{equation}
        x_1^{(0)}:=x_1,\quad x_1^{(n)}:=[x_0,x_1^{(n-1)}]\quad(n>0).
    \end{equation}
    Write $\N$ (resp. $\N^\infty$) for the set of non-negative integers (resp. the disjoint union of the $\N^r$ for all $r\geq0$). The \textit{Magnus polynomial} for a multi-index $\mathbf{k} = (k_1,\ldots,k_{r-1};k_r)\in\N^\infty\times\N$ is defined by
    \begin{equation}
        M^{(\mathbf{k})}:=x_1^{(k_1)}\cdots x_1^{(k_{r-1})}x_0^{k_r},
    \end{equation}
    as introduced in \cite{Ma37}. By considering the $\C$-linear isomorphism 
    \begin{equation}
        \pi:\C\langle X\rangle x_1\overset{\sim}{\longrightarrow}\C\langle Y\rangle
    \end{equation}
    determined by $x_0^{s_1}x_1\cdots x_0^{s_r}x_1\longmapsto y_{s_1}\cdots y_{s_r}$, we get:
    \begin{mainthm}[Theorem \ref{Magnus poly and n-fold prod}, Corollary \ref{n-fold product}]\label{main4}
        Let $\mathbf{k}=(k_1,\ldots k_{r-1};k_r)\in\N^\infty\times\N$. Then we have
        \begin{equation}\label{eq.main4}
            \polylog_{k_1}^-\cdots\polylog_{k_r}^-=\polylog^-_{\pi(M^{(\mathbf{k})}x_1)}
        \end{equation}
    \end{mainthm}
    For example, when $r=2$ (Corollary \ref{non-positive polylog shuffle}),
    \begin{equation}\label{eq.main4.r=2}
        \polylog_{k_1}^-\polylog_{k_2}^-=\sum_{i=0}^{k_1}(-1)^i\binom{k_1}{i}\polylog_{(k_1-i,k_2+i)}^-=\sum_{i=0}^{k_2}(-1)^i\binom{k_2}{i}\polylog_{(k_2-i,k_1+i)}^-.
    \end{equation}
    We derive Theorem \ref{main4} as a kind of M\"obius inversion of the following known formula.
    \begin{mainthm}[Proposition \ref{polylog and array binomial} cf. \cite{D17a} Theorem 2 p.11]\label{main1}
        Let $r>1$ and $w=y_{s_1}\cdots y_{s_r}\in Y^\ast$. Then we have
        \begin{align*}
            \polylog_w^-=&\sum_{k_1=0}^{s_1}\sum_{k_2=0}^{s_1+s_2-k_1}\cdots\sum_{k_{r-1}=0}^{(s_1+\cdots+s_{r-1})-(k_1+\cdots+k_{r-2})}\binom{s_1}{k_1}\binom{s_1+s_2-k_1}{k_2}\cdots\\
            &\cdot\binom{s_1+\cdots+s_{r-1}-k_1-\cdots-k_{r-2}}{k_{r-1}}\cdot\polylog^-_{{k_1}}\polylog^-_{{k_2}}\cdots\polylog^-_{{k_{r-1}}}\polylog^-_{{s_1+\cdots+s_r-k_1-\cdots-k_{r-1}}}.
        \end{align*}
    \end{mainthm}
    (A small typographical error in \cite{D17a} will be corrected in Proposition \ref{polylog and array binomial} of this paper.)\par
    In particular, if we define $\sigma(\mathbf{k}):=(k_{\sigma(1)},\ldots,k_{\sigma(r-1)};k_{\sigma(r)})$ for any $\mathbf{k} = (k_1, \ldots, k_{r-1}; k_r) \in \N^\infty \times \N$ and any permutation $\sigma \in \mathfrak{S}_r$, then we obtain the following result:
    \begin{mainthm}[Theorem \ref{Magnus and symmetric group}]\label{main5}
        For any $r>0$ and any $\sigma\in\mathfrak{S}_r$, we have
        \begin{equation}
            \pi\left((M^{(\mathbf{k})}-M^{(\sigma(\mathbf{k}))})x_1\right)\in\Ker\polylog^-_\bullet.
        \end{equation}
    \end{mainthm}
    This corollary has an application to $\Q$-linear functional equations of non-positive MPLs. For example we obtain the following equation:
    \begin{exa*}[Example \ref{Algorithm examples of Magnus}]
        Computing $(M^{(0,1;2)}-M^{(1,2;0)})x_1$, we obtain
        \begin{equation}
            2\polylog^-_{(0,1,2)}+2\polylog^-_{(1,1,1)}+\polylog^-_{(0,3,0)}=\polylog^-_{(0,0,3)}+\polylog^-_{(1,2,0)}+\polylog^-_{(1,0,2)}+2\polylog^-_{(0,2,1)}.
        \end{equation}
        See Section 6 for more detailed computations.
    \end{exa*}
    However, the kernel $\Ker\polylog^-_\bullet$ is not generated by $\{\pi\left((M^{(\mathbf{k})}-M^{(\sigma(\mathbf{k}))})x_1\right)\}_{\sigma,\mathbf{k}}$. This point is examined in Proposition \ref{Permutation is not equal to Ker} and will be discussed more in \cite{KKN}. 
    \subsubsection*{Outline}
    The contents of this paper are organized as follows: In Sections 2 and 3, we introduce multiple polylogarithms with non-positive multi-indices (hereafter, non-positive MPL) and set up them in connection with non-commutative polynomials. In Section 4, we first show the above Theorem \ref{main1} (Proposition \ref{polylog and array binomial}) by analysing inductive applications of $\theta_0=z\partial_z$ and of multiplication to the function $\lambda=z/(1-z)$ according to (\ref{eq.another ex. of n.p.MPL}). We then show a key formula (Theorem \ref{product between literature and word}) that represent a product of the form $\polylog^-_{y_n}\polylog^-_{y_rw'}$ as an alternating linear sum of the terms $\polylog^-_{y_{n-k}y_{r+k}w'}\ (k=0,\ldots,m)$ with binomial coefficients. Theorem \ref{main4} is then a consequence of iterative applications of the key formula (Corollary \ref{n-fold product}). Section 5 is devoted to interpretation of our main formulas in the framework of \cite{Na23} which provides alternative proofs. Finally in Section 6, we discuss applications to functional equations of non-positive MPLs and the kernel of $\polylog^-_\bullet$.
    \subsubsection*{Acknowledgements}
    This paper is partially based on my master thesis \cite{K25}. I would like to express my deepest gratitude to my advisor, Professor Hiroaki Nakamura, for his warm guidance and constant support throughout the preparation of this paper. This work was supported by JST SPRING, Grant Number JPMJSP2138.
\section{Some preliminaries}
\textit{Throughout this paper, $\N$ denotes the set of non-negative integers.}
    Let $\Omega$ be the simply connected space $\C-((-\infty,0]\cup[1,\infty))$.
    \begin{lem}
        Let $\mathcal{H}(\Omega)$ be the ring of holomorphic functions on $\Omega$. Then, the followings hold.
        \begin{enumerate}[(a)]
            \item 
            Let $\mathcal{M}$ be the ideal of holomorphic functions on $\Omega$ which vanish at 0. We define formally the operators:
            \begin{align*}
                \theta_0:=z\partial_z,\quad\iota_0(f):=\int_0^z\frac{f(t)}{t}dt,\quad\iota_1(f):=\int_0^z\frac{f(t)}{1-t}dt\quad(f\in\mathcal{H}(\Omega)).
            \end{align*}
            Then, these operators act $\C$-linearly on $\mathcal{M}$ (i.e.  $\theta_0,\iota_0,\iota_1\in\End_\C(\mathcal{M})$). In particular, the domain of $\iota_1$ can be extended to $\mathcal{H}(\Omega)$. (Note that these integration paths are included in $\Omega$.)
            \item 
            Let $\lambda:=z/(1-z)$. Then, for any $f\in\mathcal{M},\ \theta_0\iota_1(f)=\lambda f.$ 
        \end{enumerate}
    \end{lem}
    \begin{proof}
        (a) Each operator is clearly $\C$-linear. For any $f\in\mathcal{M}$, we can write the Taylor series of $f$ at 0 as follows:
        \begin{align*}
            f(z)=\sum_{n>0}c_nz^n\ (c_n\in\C).
        \end{align*}
        Hence, for $z$ in a neighbourhood of 0, we have
        \begin{align*}
            \theta_0(f(z))=\sum_{n>0}nc_nz^n,\quad\iota_0(f(z))=\sum_{n>0}\frac{c_n}{n}z^n.
        \end{align*}
        So, $\theta_0(f),\iota_0(f)\in\mathcal{M}$ follows. Since $f(z)/(1-z)=0$ holds at $z=0$, considering the Taylor series at 0, $\iota_1(f)\in\mathcal{M}$ also follows. In addition, if $f\in\mathcal{H}(\Omega)$, then we have
        \begin{align*}
            \int_0^z\frac{f(t)}{1-t}dt\in\C\cdot\log(1-z)+\mathcal{M}
        \end{align*}
        so we can extend the domain of $\iota_1$ to $\mathcal{H}(\Omega)$ because $\log(1-z)\in\mathcal{M}.$\par
        (b) If $F:=\iota_1(f)$, then we have $\theta_0\iota_1(f)=z\partial_z(F)=\lambda f.$
    \end{proof}
    Write $1_\Omega$ for the constant function on $\Omega$ valued at $1$. Then, from the above lemma, we have
    \begin{align*}
        \theta_0^{s_1+1}\iota_1\cdots\theta_0^{s_r+1}\iota_1(1_\Omega)\in\mathcal{M}\quad(s_1,\ldots,s_r\in\N). 
    \end{align*}
    Now, we consider Taylor series of $\theta_0^{p+1}\iota_1(1_\Omega)$ for any $p\in\N$ to obtain the following lemma.
    \begin{lem}
        For any $p\in\N$, we have:
        \begin{equation}
            \theta_0^{p+1}\iota_1(1_\Omega)=\sum_{n>0}n^pz^n,\quad \theta_0\iota_1\qty(\sum_{n>0}n^pz^n)=\sum_{r>n>0}n^pz^r.
        \end{equation}
        In particular, for $s_1,\ldots,s_r\in\N$,
        \begin{align}\label{eq.operator and polylog}
            \theta_0^{s_1+1}\iota_1\cdots\theta_0^{s_r+1}\iota_1(1_\Omega)=\sum_{n_1>\cdots>n_r>0}n_1^{s_1}\cdots n_r^{s_r}z^{n_1}.
        \end{align}
    \end{lem}
    \begin{proof}
    The proof of the claim is obtained from the computation:
        \begin{align*}
            \theta_0^{p+1}\iota_1(1_\Omega)=\theta_0^p(\lambda)=\theta_0^{p-1}&\left(\sum_{n>0}\theta_0(z^n)\right)=\theta_0^{p-1}\qty(\sum_{n>0}nz^n)=\cdots=\sum_{n>0}n^pz^n,\\
            \theta_0\iota_1\qty(\sum_{n>0}n^pz^n)&=\qty(\sum_{m>0}z^m)\qty(\sum_{n>0}n^pz^n)\\
            &=\lim_{m,n\to+\infty}\sum_{i=1}^m\sum_{k=1}^nk^pz^{k+i}\\
            &=\lim_{m,n\to+\infty}\sum_{r=2}^{m+n}\left(\sum_{r>i>0}(r-i)^pz^r\right)\\
            &=\lim_{m,n\to+\infty}\sum_{r=2}^{m+n}\qty(\sum_{r>i>0}i^pz^r)\\
            &=\sum_{r>n>0}n^pz^r.
        \end{align*}
        By calculating inductively, we get the equation (\ref{eq.operator and polylog}).
    \end{proof}
    Then, we define the multiple polylogarithms at non-positive indices on $D$. By the above lemma, this is the convergent series function on $D$.
    \begin{defi}
        Let $\mathbf{s}:=(s_1,\ldots,s_r)\in\N^r$ and $r\geq 0$. Then:
        \begin{align*}
            \polylog^-_{\mathbf{s}}(z):=\polylog_{-s_1,\ldots,-s_r}(z)=\sum_{n_1>\cdots>n_r>0}n_1^{s_1}\cdots n_r^{s_r}z^{n_1}\quad(z\in D)
        \end{align*}
        where if $\mathbf{s}=\ast$, we promise $\polylog^-_{\mathbf{s}}=1$.
    \end{defi}
    We easily see that this series is analytically continued to $\Omega$, i.e. $\polylog^-_{\mathbf{s}}(z)=\theta_0^{s_1+1}\iota_1\cdots\theta_0^{s_r+1}\iota_1(1_\Omega)$ on $\Omega$. In particular, the right hand side of (\ref{eq.operator and polylog}) shows us:
    \begin{equation}
        \polylog^-_\mathbf{s}(z)\in\Q[z,(1-z)^{-1}].
    \end{equation}
    Therefore, we regard $\polylog^-_\mathbf{s}:=\polylog^-_{\mathbf{s}}(z)$ as a $\Q$(or $\C$)-rational function in this paper.
\section{Words, free algebras and Magnus polynomials}
    We prepare combinatorial properties of MPL. Let $X:=\{x_0,x_1\},Y:=\{y_0,y_1,\ldots,y_n,\ldots\}_{n\in\N}$, and let $X^\ast,Y^\ast$ be free monoids respectively generated by $X,Y$.
    \begin{notation}
        Let $\N^\infty:=\bigsqcup_{r>0}\N^r(\textit{where }\N^0=\{\ast\})$ be the set of indices.
    \end{notation}
    \begin{lem}
        The map
        \begin{equation}
            \varphi_Y:\N^\infty\to Y^\ast;(s_1,\ldots,s_r)\mapsto y_{s_1}\cdots y_{s_r}
        \end{equation}
        is bijective. Here for the empty word $\epsilon_Y\in Y^\ast$, $\varphi_Y(\ast)=\epsilon_Y$.
        %if $\epsilon_Y\in Y^\ast$, then $\varphi_Y(\ast)=\epsilon_Y$.
    \end{lem}
    \begin{proof}
        It is clear by the definition of $\varphi_Y$.
    \end{proof}
    Let $\psi_Y$ be the inverse map of $\varphi_Y$. Then, for $w\in Y^\ast$, we define
    \begin{equation}
        \polylog_w:=\left\{
        \begin{array}{cl}
            \polylog_{\psi_Y(w)} &(\textit{if }w\neq \epsilon_Y)  \\
            1 &(\textit{if }w=\epsilon_Y) 
        \end{array}
        \right.,\quad\polylog^-_w:=\left\lbrace
        \begin{array}{cl}
            \polylog^-_{\psi_Y(w)} &(\textit{if }w\neq\epsilon_Y)  \\
            1 &(\textit{if }w=\epsilon_Y) 
        \end{array}
        \right..
    \end{equation}
    \begin{exa}
        \begin{enumerate}[(1)]
            \item 
            $\polylog_{y_1y_2y_3}=\polylog_{(1,2,3)}=\sum_{n_1>n_2>n_3>0}n_1^{-1}n_2^{-2}n_3^{-3}z^{n_1}.$
            \item 
            $\polylog^-_{y_4y_5y_6}=\polylog^-_{(4,5,6)}=\sum_{n_1>n_2>n_3>0}n_1^4n_2^5n_3^6z^{n_1}.$
        \end{enumerate}
    \end{exa}
    \begin{lem}
        Let $\epsilon_X$ be an empty word in $X^\ast$. Then,
        \begin{equation}
            p:Y^\ast\to(X^\ast x_1)\cup\{\epsilon_X\};y_{s_1}\cdots y_{s_r}\mapsto x_0^{s_1}x_1\cdots x_0^{s_r}x_1
        \end{equation}
        is a monoid isomorphism.
    \end{lem}
    \begin{proof}
        It is clear by the definition of $p$.
    \end{proof}
    Let $R$ be a commutative integral domain of characteristic $0$ and $R\langle X\rangle=R\langle x_0,x_1\rangle$(respectively $R\langle Y\rangle=R\langle y_0,y_1,\ldots\rangle$) be a unital free associative algebra generated by $x_0,x_1$ on $R$ (resp. generated by $y_0,y_1,\ldots$ on $R$).
    \begin{defi}
        We define an $R$-algebra homomorphism $\pi:R\langle X\rangle x_1\to R\langle Y\rangle$ as the $R$-algebra extension of the mapping $p^{-1}:x_0^kx_1\mapsto y_k$.
    \end{defi}
    \begin{lem}\label{index algebra}
        $\pi$ is an isomorphism.
    \end{lem}
    \begin{proof}
        It is clear from the definition of $\pi$ and the above lemma.
    \end{proof}
    For any $u,v\in R\langle X\rangle$, we define Lie bracket $[u,v]:=uv-vu$, and we define $x_1^{(n)}\in R\langle X\rangle$ as follows:
    \begin{equation}
        x_1^{(0)}:=x_1,\quad x_1^{(n+1)}:=[x_0,x_1^{(n)}]\quad(n=0,1,2,\ldots).
    \end{equation}
    Then, the following holds.
    \begin{lem}\textup{(\cite[(5),(6)]{Na23}, \cite[(4)]{Ma37})}\ 
        For $n\in\N$,
        \begin{equation}
            x_1^{(n)}=\sum_{k=0}^n(-1)^k\binom{n}{k}x_0^{n-k}x_1x_0^k.
        \end{equation}
    \end{lem}
    \begin{proof}
        We use induction on $n$. The case for $n=0$ is clear. Let $n>0$ and we assume that the case for $n-1$ is true. Then, we complete the proof for $n$ by the calculation:
        \begin{align*}
            x_1^{(n)}&=x_0x_1^{(n-1)}-x_1^{(n-1)}x_0\\
            &=\sum_{k=0}^{n-1}(-1)^k\binom{n-1}{k}x_0x_0^{n-1-k}x_1x_0^k+\sum_{k=0}^{n-1}(-1)^{k+1}\binom{n-1}{k}x_0^{n-1-k}x_1x_0^kx_0\\
            &=\sum_{k=0}^{n-1}(-1)^k\binom{n-1}{k}x_0^{n-k}x_1x_0^k+\sum_{k=1}^{n}(-1)^{k}\binom{n-1}{k-1}x_0^{n-k}x_1x_0^k\\
            &=x_0^nx_1+\sum_{k=1}^{n-1}(-1)^k\left(\binom{n-1}{k}+\binom{n-1}{k-1}\right)x_0^{n-k}x_1x_0^k+(-1)^nx_1x_0^n\\
            &=\sum_{k=0}^n(-1)^k\binom{n}{k}x_0^{n-k}x_1x_0^k.
        \end{align*}
    \end{proof}
    We define the Magnus polynomials $M^{(\mathbf{k})}\in R\langle X\rangle$ for $\mathbf{k}=(k_1,\ldots,k_n;k_\infty)\in\N^\infty\times\N$ as follows:
    \begin{defi}[Magnus polynomial, \cite{Ma37}]
        \begin{equation}
            M^{(\mathbf{k})}:=x_1^{(k_1)}\cdots x_1^{(k_n)}x_0^{k_\infty}\in R\langle X\rangle.
        \end{equation}
        Here, let us conventionally write $(\ast,k_\infty)\in\N^0\times\N\subset\N^\infty\times\N$ as $(;k_\infty)$.
    \end{defi}
    \begin{rem}
        W. Magnus proved that the set of all Magnus polynomials is the basis of $R\langle X\rangle$, which is the main result of \cite{Ma37}.
    \end{rem}
    Then the following holds from the above lemma.
    \begin{cor}\label{expantion of Magnus poly}
        \begin{equation}
            M^{(\mathbf{k})}=\sum_{i_1=0}^{k_1}\cdots\sum_{i_n=0}^{k_n}\left(\prod_{j=1}^n(-1)^{i_j}\binom{k_j}{i_j}\right)x_0^{k_1-i_1}x_1x_0^{k_2-i_2+i_1}x_1\cdots x_0^{k_n-i_n+i_{n-1}}x_1x_0^{k_\infty+i_n}\in R\langle X\rangle.
        \end{equation}
    \end{cor}
\section{Proofs of Theorem \ref{main4} and Theorem \ref{main1}}
    We begin by proving Theorem \ref{main1} of Introduction.
    \begin{prop}[\cite{D17a} p.175]\label{polylog and array binomial}
        For $r>1,w=y_{s_1}\cdots y_{s_r}\in Y^\ast$,
        \begin{align*}
            \polylog_w^-=&\sum_{k_1=0}^{s_1}\sum_{k_2=0}^{s_1+s_2-k_1}\cdots\sum_{k_{r-1}=0}^{(s_1+\cdots+s_{r-1})-(k_1+\cdots+k_{r-2})}\binom{s_1}{k_1}\binom{s_1+s_2-k_1}{k_2}\cdots\\
            &\binom{s_1+\cdots+s_{r-1}-k_1-\cdots-k_{r-2}}{k_{r-1}}\polylog^-_{y_{k_1}}\polylog^-_{y_{k_2}}\cdots\polylog_{y_{k_{r-1}}}^-\polylog^-_{y_{s_1+s_2+\cdots+s_r-k_1-\cdots-k_{r-1}}}.
        \end{align*}
    \end{prop}
    \begin{proof}
        First, note that $\theta_0(fg)=(\theta_0f)g+f(\theta_0g)$ for any $f,g\in\C[z]$, we have
        \begin{align*}
            \polylog_{y_{s_1}^-\cdots y_{s_r}}^-&=(\theta_0^{s_1+1}\iota_1\cdots\theta_0^{s_r+1}\iota_1)1_\Omega\\
            &=\sum_{k_1=0}^{s_1}\binom{s_1}{k_1}\polylog_{y_{k_1}}^-\polylog_{y_{s_1+s_2-k_1}y_{s_3}\cdots y_{s_r}}^-.
        \end{align*}
        Then, for any $r>1$, we obtain by inductively applying above 
        \begin{align*}
            \polylog_{y_{s_1}\cdots y_{s_r}}^-&=\sum_{k_1=0}^{s_1}\binom{s_1}{k_1}\polylog_{y_{k_1}}^-\polylog_{y_{s_1+s_2-k_1}y_{s_3}\cdots y_{s_r}}^-\\
            &=\sum_{k_1=0}^{s_1}\sum_{k_2=0}^{s_1+s_2-k_1}\binom{s_1}{k_1}\binom{s_1+s_2-k_1}{k_2}\polylog_{y_{k_1}}^-\polylog_{y_{k_2}}^-\polylog_{y_{s_1+s_2+s_3-k_1-k_2}y_{s_4}\cdots y_{s_r}}^-\\
            &=\cdots\\
            &=\sum_{k_1=0}^{s_1}\sum_{k_2=0}^{s_1+s_2-k_1}\cdots\sum_{k_{r-1}=0}^{(s_1+\cdots+s_{r-1})-(k_1+\cdots+k_{r-2})}\binom{s_1}{k_1}\binom{s_1+s_2-k_1}{k_2}\cdots\\
            &\binom{s_1+\cdots+s_{r-1}-k_1-\cdots-k_{r-2}}{k_{r-1}}\polylog^-_{y_{k_1}}\polylog^-_{y_{k_2}}\cdots\polylog_{y_{k_{r-1}}}^-\polylog^-_{y_{s_1+s_2+\cdots+s_r-k_1-\cdots-k_{r-1}}}.
        \end{align*}
    \end{proof}
    \begin{rem}
        The equation in this proposition is appearing in the proof of \cite[Theorem 2]{D17a} with some typographic error, which correct in the above.
    \end{rem}
    \begin{cor}
        For $n,r\in\N$, $\polylog_{y_0^ny_r}^-=(\polylog^-_{y_0})^n\polylog^-_{y_r}.$
    \end{cor}
    \begin{exa}
        \begin{enumerate}[(1)]
            \item 
            $\polylog^-_{y_0^r}=(\polylog^-_{y_0})^r=\displaystyle\left(\frac{z}{1-z}\right)^r.$
            \item 
            $\displaystyle\polylog^-_{y_1^2}=\polylog^-_{y_0}\polylog^-_{y_2}+(\polylog^-_{y_1})^2=\left(\frac{z}{1-z}\right)\left(\frac{z+z^2}{(1-z)^3)}\right)+\left(\frac{z}{(1-z)^2}\right)^2=\frac{2z^2+z^3}{(1-z)^4}.$
        \end{enumerate}
    \end{exa}
    From the above examples, for $n\in\N$ we have
    \begin{equation}
        \polylog_{y_1}^-\polylog_{y_n}^-=\polylog_{y_1y_n}^--\polylog_{y_0y_n}^-.
    \end{equation}
    Then we apply $\theta_0$ to both side, we have
    \begin{align*}
        \polylog_{y_2}^-\polylog_{y_n}^-&=\polylog_{y_2y_n}^--\polylog_{y_1y_{n+1}}^--\polylog_{y_1}^-\polylog_{y_{n+1}}^-\\
        &=\polylog_{y_2y_n}^--\polylog_{y_1y_{n+1}}^--(\polylog_{y_1y_{n+1}}^--\polylog_{y_0y_{n+2}}^-)\\
        &=\polylog^-_{y_2y_n}-2\polylog_{y_1y_{n+1}}^-+\polylog_{y_0y_{n+2}}^-.
    \end{align*}
    By calculating inductively, we obtain the following.
    \begin{theo}\label{product between literature and word}
        For $m,r\in\N,w=y_rw'\in Y^\ast$,
        \begin{equation}
            \polylog_{y_m}^-\polylog_w^-=\sum_{k=0}^m(-1)^k\binom{m}{k}\polylog_{y_{m-k}y_{r+k}w'}^-.
        \end{equation}
    \end{theo}
    \begin{proof}
        We will show (left)$=$(center) by induction on $m$. It is clear that $\polylog_{y_0}^-\polylog_{w}^-=\polylog_{y_0w}^-$ holds. Let $m>0$ and we assume that the desired equation holds for $m-1$. Then for any $r\geq0$ and $w=y_rw'\in Y^\ast\setminus\{\epsilon\}$,
        \begin{align*}
            \theta_0(\polylog_{y_{m-1}}^-\polylog_{w}^-)&=\polylog_{y_m}^-\polylog_{w}^-+\polylog_{y_{m-1}}\polylog_{y_{r+1}w'}^-\\
            &=\sum_{k=0}^{m-1}(-1)^k\binom{m-1}{k}\theta_0(\polylog_{y_{m-1-k}y_{r+k}w'}^-)=\sum_{k=0}^{m-1}(-1)^k\binom{m-1}{k}\polylog_{y_{m-k}y_{r+k}w'}^-.
        \end{align*}
        Therefore, by the assumption of induction, we obtain
        \begin{align*}
            \polylog_{y_m}^-\polylog_{w}^-&=\sum_{k=0}^{m-1}(-1)^k\binom{m-1}{k}\polylog_{y_{m-k}y_{r+k}w'}^--\polylog_{y_{m-1}}\polylog_{y_{r+1}w'}^-\\
            &=\sum_{k=0}^{m-1}(-1)^k\binom{m-1}{k}\polylog_{y_{m-k}y_{r+k}w'}^--\sum_{k=0}^{m-1}(-1)^k\binom{m-1}{k}\polylog_{y_{m-1-k}y_{r+1+k}w'}^-\\
            &=\polylog_{y_mw}^-+\sum_{k=1}^{m-1}(-1)^k\left(\binom{m-1}{k}+\binom{m-1}{k-1}\right)\polylog_{y_{m-k}y_{r+k}w'}^-+(-1)^m\polylog_{y_0y_{n+m}w'}^-\\
            &=\sum_{k=0}^m(-1)^k\binom{m}{k}\polylog_{y_{m-k}y_{r+k}w'}^-.
        \end{align*}
    \end{proof}
    In particular, when $w\in Y$, the following holds.
    \begin{cor}\label{non-positive polylog shuffle}
        For $m,n\in\N$,
        \begin{align*}
            \polylog_{y_m}^-\polylog_{y_n}^-&=\sum_{k=0}^m(-1)^k\binom{m}{k}\polylog_{y_{m-k}y_{n+k}}^-=\sum_{k=0}^n(-1)^k\binom{n}{k}\polylog_{y_{n-k}y_{m+k}}^-
        \end{align*}
    \end{cor}
    In addition, applying this theorem inductively, we can calculate $r$-fold product of $\polylog^-_{y_s}$.
    \begin{cor}\label{n-fold product}
        For $s_1,\ldots,s_n\in\N(n>1)$,
        \begin{equation}
            \prod_{i=1}^n\polylog_{y_{s_i}}^-=\sum_{k_1=0}^{s_1}\sum_{k_2=0}^{s_2}\cdots\sum_{k_{n-1}=0}^{s_{n-1}}\left(\prod_{i=1}^{n-1}(-1)^{k_i}\binom{s_i}{k_i}\right)\polylog_{y_{s_1-k_1}y_{s_2-k_2+k_1}\cdots y_{s_{n-1}-k_{n-1}+k_{n-2}}y_{s_n+k_{n-1}}}^-.
        \end{equation}
    \end{cor}
    \begin{proof}
        We use an induction on $n$. When $n=2$, Corollary \ref{non-positive polylog shuffle} imply the assertion. Now, to prove the case in $n>2$, we assume that the assertion holds for $n-1$. Then, we have
        \begin{align*}
            \prod_{i=1}^n\polylog_{y_{s_i}}^-&=\polylog_{y_{s_1}}^-\left(\prod_{i=2}^n\polylog_{y_{s_i}}^-\right)\\
            &=\sum_{k_2=0}^{s_2}\cdots\sum_{k_{n-1}=0}^{s_{n-1}}\left(\prod_{i=2}^{n-1}(-1)^{k_i}\binom{s_i}{k_i}\right)\polylog_{y_{s_1}}^-\polylog_{y_{s_2-k_2}y_{s_3-k_3+k_2}\cdots y_{s_n+k_{n-1}}}^-.
        \end{align*}
        Now, from Theorem \ref{product between literature and word}, we see
        \begin{equation}
            \polylog_{y_{s_1}}^-\polylog_{y_{s_2-k_2}y_{s_3-k_3+k_2}\cdots y_{s_n+k_{n-1}}}^-=\sum_{k_1=0}^{s_1}(-1)^{k_1}\binom{s_1}{k_1}\polylog_{y_{s_1-k_1}y_{s_2-k_2+k_1}\cdots y_{s_n-k_{n-1}}}^-.
        \end{equation}
        So we obtain the desired result.
    \end{proof}
    Now, let us prove Theorem \ref{main4} of Introduction:
    \begin{theo}\label{Magnus poly and n-fold prod}
        For $\mathbf{k}=(k_1,\ldots,k_r;k_\infty)\in\N^\infty\times\N$
        \begin{equation}
            \polylog_{k_1}^-\cdots\polylog_{k_r}^-\polylog_{k_\infty}^-=\polylog^-_{\pi(M^{(\mathbf{k})}x_1)}.
        \end{equation}
    \end{theo}
    \begin{proof}
        Compare Corollary \ref{expantion of Magnus poly} with Corollary \ref{n-fold product} to complete this proof.
    \end{proof}
    Observe in the above that the coefficients appearing in Corollary \ref{expantion of Magnus poly} and in Corollary \ref{n-fold product} coincide with each others. In the next section, we will discuss a combinatorial framework due to \cite{Na23} that explains this coincidence from the viewpoint of a duality between multiple binomial coefficients.
\section{Reformulation of main theorems}
    Let $\Q\langle Y\rangle$ be a vector space freely generated by $Y^\ast$ over $\Q$. Then, we define the following $\Q$-linear map (cf. \cite{D17a}):
    \begin{equation}
        \polylog_\bullet^-:\Q\langle Y\rangle\longrightarrow\Q\qty[z,\frac{1}{1-z}]\ ;\ w\longmapsto\polylog_w^-.
    \end{equation}
    \begin{defi}
        For $\mathbf{k}\in\N^\infty(\times\N)$, there is a unique $r\in\N$ such that $\mathbf{k}\in\N^r\times\N\subset\N^\infty\times\N$ because of the definition of $\N^\infty$. Then we define the \textup{depth} of an index, denoted by $\depth\mathbf{k}$, as the number $r$, and we define $|\bf{k}|\in\Z_{\geq 0}$ as a sum of each coefficients of $\mathbf{k}$.
    \end{defi}
    Then for simplicity, we introduce the following convenient notations.
    \begin{defi}[Array binomial coefficient (\cite{Na23})]
        Let $\mathbf{s}:=(s_1,\ldots,s_{r_1};s_\infty),\mathbf{k}:=(k_1,\ldots,k_{r_2};k_\infty)\in\N^\infty\times\N\,(r_1,r_2>1)$. Then we define the followings.\par
        (1)An \textup{array binomial coefficient} for $\mathbf{s,k}$ denote as
        \begin{equation}
            \binom{\bf{s}}{\bf{k}}:=\binom{s_1}{k_1}\binom{s_1+s_2-k_1}{k_2}\cdots\binom{s_1+\cdots+s_r-k_1-\cdots-k_{r-1}}{k_r}.
        \end{equation}
        Here $r=\depth\bf{s}$, and we define the value as $0$ if even one of the following conditions does not hold.
        \begin{enumerate}[(a)]
            \item 
            $\depth\mathbf{s}=\depth\mathbf{k}(=r),$
            \item 
            For $n=1,\ldots,r$, $\sum_{i=1}^n(s_i-k_i)\geq0,$
            \item 
            $|\bf{s}|=|\bf{k}|.$
        \end{enumerate}\par
        (2)If $r_1=r_2=:d$, then a \textup{dual array binomial coefficient} for $\mathbf{s,k}$ denote as
        \begin{equation}
            \duarrbinom{\mathbf{s}}{\mathbf{k}}:=(-1)^{\sum_{i=1}^{d}(d-i+1)(s_i-k_i)}\binom{s_1}{s_1-k_1}\binom{s_2}{s_1+s_2-k_1-k_2}\cdots\binom{s_d}{\sum_{i=1}^d(s_i-k_i)}.
        \end{equation}
        Here we define the value as $0$ if one of the above conditions (a)-(c) does not hold.
    \end{defi}
    Indeed, it is known that these satisfy a certain duality.
    \begin{prop}\textup{(\cite[Remark 2.8]{Na23})}
        For $\mathbf{s,k}\in\N^\infty\times
        \N$, we have:
        \begin{equation}
            \sum_{\mathbf{u}\in\N^\infty\times\N}\binom{\mathbf{s}}{\mathbf{u}}\duarrbinom{\mathbf{u}}{\mathbf{k}}=\delta^{\mathbf{k}}_{\mathbf{s}}.
        \end{equation}
        Here, $\delta^{\mathbf{k}}_{\mathbf{s}}$ is Kronecker symbol, i.e. designating $0$ or $1$ according to whether $\mathbf{s}\neq\mathbf{k}$ or $\mathbf{s}=\mathbf{k}$ respectively.
    \end{prop}
    For example of array binomial coefficients, if $w=y_{s_1}\cdots y_{s_r}$, then for $\mathbf{s}=(s_1,\ldots,s_{r-1};s_r)\in\N^\infty\times\N$ we can rewrite the equation of Proposition \ref{polylog and array binomial} more simply:
    \begin{equation}\label{eq.expansion of polylog}
        \polylog_{w}^-=\sum_{\bf{k}\in\N^\infty\times\N}\binom{\bf{s}}{\bf{k}}\polylog^-_{k_1}\cdots\polylog_{k_{r-1}}^-\polylog_{k_\infty}^-.
    \end{equation}
    Here $k_\infty=s_1+\cdots+s_r-k_1-\cdots-k_{r-1}$. This enables us to reprove Proposition \ref{polylog and array binomial} (Theorem \ref{main1}):
    \begin{cor}\label{Magnus poly and non-positive multi-index polylog}
        Let $w=y_{s_1}\cdots y_{s_r}\in Y^\ast$ and $\mathbf{s}:=(s_1,\ldots,s_{r-1};s_r)\in\N^\infty\times\N$. Then we obtain:
        \begin{equation}
            \polylog_w^-=\sum_{\bf{k}\in\N^\infty\times\N}\binom{\bf{s}}{\bf{k}}\polylog^-_{\pi(M^{(\mathbf{k})}x_1)}.
        \end{equation}
    \end{cor}
    \begin{proof}
        It is clear from (\ref{eq.expansion of polylog}) and Theorem \ref{Magnus poly and n-fold prod}.
    \end{proof}
    Now, for $\mathbf{s}=(s_1,\ldots,s_{r-1};s_r)\in\N^\infty\times\N$, set $w_{\mathbf{s}}:=x_0^{s_1}x_1\cdots x_0^{s_{r-1}}x_1x_0^{s_r}$. H. Nakamura derived the equations between words $w_\mathbf{s}$ and Magnus polynomials like a kind of the M\"obius inversion formula.
    \begin{prop}[\cite{Na23}]\label{Magnus expression}
        \begin{equation}
            w_{\mathbf{s}}=\sum_{\mathbf{k}\in\N^\infty\times\N}\binom{\bf{s}}{\bf{k}}M^{(\mathbf{k})},\qquad M^{(\mathbf{k})}=\sum_{\mathbf{s}\in\N^\infty\times\N}\duarrbinom{\mathbf{k}}{\mathbf{s
            }}w_\mathbf{s}.
        \end{equation}
    \end{prop}
    \begin{proof}
        By $u=w_{\mathbf{s}}$ in \cite[Corollary 2.9]{Na23}, we get the first equation. The second equation is shown in \cite[Remark 2.8]{Na23}.
    \end{proof}
    \begin{proof}
        (Another proof of Corollary \ref{Magnus poly and non-positive multi-index polylog})\par
        On the equation in the assertion of Proposition \ref{Magnus expression}, concatenate $x_1$ at last in the word and take the image by $\pi$ to have $\pi(w_\mathbf{s}x_1)=y_{s_1}\cdots y_{s_r}(=:w)$ on the left hand side. Now, $\pi$ is a $\C$-isomorphism, so
        \begin{equation}
            w=\sum_{\mathbf{k}\in\N^\infty\times\N}\binom{\mathbf{s}}{\mathbf{k}}\pi(M^{(\mathbf{k})}x_1).
        \end{equation}
        Hence we obtain the equation by applying a $\C$-linear map $\polylog_\bullet^-$ to both side.
    \end{proof}
\section{Some applications and problems}
    Recall that the composed map 
    \begin{equation}
        \polylog^-_\bullet\circ\pi:\C\langle X\rangle x_1\overset{\pi}{\longrightarrow}\C\langle Y\rangle\overset{\polylog^-_\bullet}{\longrightarrow}\C\qty[z,\frac{1}{1-z}]
    \end{equation}
    sends the Magnus polynomial $M^{(k_1,\ldots,k_d;k_\infty)}x_1$ to the product $\polylog^-_{k_1}(z)\cdots\polylog^-_{k_d}(z)\polylog^-_{k_\infty}(z)$. For any $\mathbf{k}=(k_1,\ldots,k_{r-1};k_r)$ and any $r$-permutation $\sigma\in\mathfrak{S}_r$, we define $\sigma(\mathbf{k}):=(k_{\sigma(1)},\ldots,k_{\sigma(r-1)};k_{\sigma(r)})$. 
    By virtue of Theorem \ref{Magnus poly and n-fold prod}, the permutations of indices of Magnus polynomials leave the corresponding $\polylog^-_\bullet$ invariant. In particular, we have the following theorem:
    \begin{theo}[Theorem \ref{main5} of Introduction]\label{Magnus and symmetric group}
        If $r>0,\sigma\in\mathfrak{S}_r$, then $\pi((M^{(\mathbf{k})}-M^{(\sigma(\mathbf{k}))})x_1)\in\Ker\polylog^-_\bullet$ holds.
    \end{theo}
    \begin{proof}
        We can easily prove this equation from Theorem \ref{Magnus poly and n-fold prod} and commutativity of $\C[z,(1-z)^{-1}]$:
        \begin{equation}
            \polylog^-_{\pi(M^{(\mathbf{k})}x_1)}=\polylog^-_{k_1}\cdots\polylog^-_{k_r}=\polylog^-_{k_{\sigma(1)}}\cdots\polylog^-_{k_{\sigma(r)}}=\polylog^-_{\pi(M^{(\sigma(\mathbf{k}))}x_1)}.
        \end{equation}
    \end{proof}
    Theorem \ref{Magnus and symmetric group} provides us with a systematic way to produce $\Q$-linear functional equations among a family of non-positive MPLs with same weight and depth. We will see some explicit examples:
    \begin{exa}\label{Algorithm examples of Magnus}
        \begin{enumerate}[(1)]
            \item 
            $\polylog^-_{\pi(M^{(1;2)}x_1-M^{(2;1)}x_1)}=0$ holds. Then let us calculate each Magnus polynomial term:
            \begin{enumerate}[$\bullet$]
                \item 
                $M^{(1;2)}x_1=[x_0,x_1]x_0^2x_1=x_0x_1x_0^2x_1-x_1x_0^3x_1.$
                \item 
                $M^{(2;1)}x_1=[x_0,[x_0,x_1]]x_0x_1=x_0^2x_1x_0x_1-2x_0x_1x_0^2x_1+x_1x_0^3x_1.$
            \end{enumerate}
            So, $M^{(1;2)}x_1-M^{(2;1)}x_1=3x_0x_1x_0^2x_1-2x_1x_0^3x_1-x_0^2x_1x_0x_1$, and this means that the following equation holds:
            \begin{equation}
                3\polylog^-_{(1,2)}=2\polylog^-_{(0,3)}+\polylog^-_{(2,1)}.
            \end{equation}
            \item 
            $\polylog^-_{\pi((M^{(0,1;2)}-M^{(1,2;0)})x_1)}=0$ holds. Then:
            \begin{itemize}
                \item 
                $M^{(0,1;2)}x_1=x_1[x_0,x_1]x_0^2x_1$\\
                $=x_1x_0x_1x_0^2x_1-x_1^2x_0^3x_1$
                \item 
                $M^{(1,2;0)}x_1=[x_0,x_1][x_0,[x_0,x_1]]x_1$\\
                $=x_0x_1x_0^2x_1^2-2x_0x_1x_0x_1x_0x_1+x_0x_1^2x_0^2x_1-x_1x_0^3x_1^2+2x_1x_0^2x_1x_0x_1-x_1x_0x_1x_0^2x_1$
            \end{itemize}
            So, $(M^{(0,1;2)}-M^{(1,2;0)})x_1=2x_1x_0x_1x_0^2x_1-x_1^2x_0^3x_1-x_0x_1x_0^2x_1^2+2x_0x_1x_0x_1x_0x_1-x_0x_1^2x_0^2x_1+x_1x_0^3x_1^2-2x_1x_0^2x_1x_0x_1$,\\
            and this means that the following equation holds:
            \begin{equation}
                2\polylog^-_{(0,1,2)}+2\polylog^-_{(1,1,1)}+\polylog^-_{(0,3,0)}=\polylog^-_{(0,0,3)}+\polylog^-_{(1,2,0)}+\polylog^-_{(1,0,2)}+2\polylog^-_{(0,2,1)}.
            \end{equation}
        \end{enumerate}
    \end{exa}
    \medskip
    Recall that $x_1$ appears exactly $\depth(\mathbf{k})$ times in each term of $M^{(\mathbf{k})}$ (cf. Corollary \ref{expantion of Magnus poly}). Since Magnus polynomials form a basis of $R\langle X\rangle$, we have $R$-module grading $\mathbf{MB}^r:=\{M^{(\mathbf{k})}\in R\langle X\rangle\mid\depth(\mathbf{k})=r-1\}$ for $r>0$, and we can express
    \begin{equation}\label{Magnus grading}
        R\langle X\rangle=\bigoplus_{r>0}\langle\mathbf{MB}^r\rangle_R.
    \end{equation}
    For any $r>0$, we define
    \begin{equation}
        S_r:=\langle\{M^{(\mathbf{k})}-M^{(\sigma(\mathbf{k}))}\mid\mathbf{k}\in\N^\infty\times\N,\depth(\mathbf{k})=r-1,\sigma\in\mathfrak{S}_{r}\}\rangle_R\subset\langle\mathbf{MB}^r\rangle_R\subset R\langle X\rangle
    \end{equation}
    and $\mathcal{S}:=\bigoplus_{r>0}S_r$. Finally we shall discuss a refinement of Theorem \ref{main5}.
    \begin{prop}\label{Permutation is not equal to Ker}
        Let $R=\C$. Then the following relation naturally holds:
        \begin{equation}
            \langle\{\pi((M^{(\mathbf{k})}-M^{(\sigma(\mathbf{k}))})x_1)\mid\mathbf{k}\in\N^\infty\times\N,\sigma\in\mathfrak{S}_{\depth(\mathbf{k})}\}\rangle_\C\subsetneq\Ker\polylog^-_\bullet.
        \end{equation}
    \end{prop}
    Before giving the proof, let us note the following equation:
    \begin{equation}
        \pi(\mathcal{S}x_1)=\langle\{\pi((M^{(\mathbf{k})}-M^{(\sigma(\mathbf{k}))})x_1)\mid\mathbf{k}\in\N^\infty\times\N,\sigma\in\mathfrak{S}_{\depth(\mathbf{k})}\}\rangle_\C.
    \end{equation}
    Thus, the above proposition is equivalent to saying $\pi(\mathcal{S}x_1)\subsetneq\Ker\polylog^-_\bullet$.
    \begin{proof}[Proof of Proposition \ref{Permutation is not equal to Ker}]
        Inclusion is clear by definition and Theorem \ref{Magnus and symmetric group}. We shall show the non-equality by using the computational result due to Duchamp et al. in \cite[Theorem 5 \& Example 7, p.181-182]{D17a}:
        \begin{equation}\label{counter example-1}
            \polylog^-_{y_5}\polylog^-_{y_4}=\displaystyle-\frac{1}{60}\polylog_{y_2}^-+\frac{1}{63}\polylog^-_{y_4}+\frac{1}{1260}\polylog_{y_{10}}^-.
        \end{equation}
        (See Example \ref{non-trivial element in kernel} for some examples including this equation.)
        The left hand side of (\ref{counter example-1}) is $\polylog_{\pi(M^{(5;4)}x_1)}^-$ by Theorem \ref{Magnus poly and n-fold prod}. Now, $y_2,y_4$ and $y_{10}$ in the right hand side correspond to $M^{(;2)},M^{(;4)}$ and $M^{(;10)}$ respectively. So we have
        \begin{equation}
            \pi\left(\left(M^{(5;4)}+\frac{1}{60}M^{(;2)}-\frac{1}{63}M^{(;4)}-\frac{1}{1260}M^{(;10)}\right)x_1\right)\in\Ker\polylog_\bullet^-.
        \end{equation}
        Therefore, by Remark \ref{Magnus grading}, all of the components spanned by $M^{(;2)}, M^{(;4)}$ and $M^{(;10)}$ are contained in $\langle\mathbf{MB}^1\rangle$, whereas $M^{(5;4)}\in\mathbf{MB}^2$. Then, we will prove $\pi(\mathcal{S}x_1)\neq\Ker\polylog^-_\bullet$ by contradiction. Here, we assume that $\pi(\mathcal{S}x_1)=\Ker\polylog_\bullet^-$. Then the following must hold:
        \begin{equation}
            M^{(5;4)}\in S_2,\quad\frac{1}{60}M^{(;2)}-\frac{1}{63}M^{(;4)}-\frac{1}{1260}M^{(;10)}\in S_1.
        \end{equation}
        Now, we see $S_0=0$ and the fact that the Magnus polynomials $M^{(;2)}, M^{(;4)}$ and $M^{(;10)}$ are linearly independent over $R$. However, if one of them holds, then the other does not. Hence this is contradiction.
    \end{proof}
    \begin{exa}\label{non-trivial element in kernel}
        In addition to the above (\ref{counter example-1}) Duchamp et al. (\cite[Example 7]{D17a}) computed the following remarkable results:
        \begin{align}
            \polylog^-_{y_6}\polylog^-_{y_7}&=-\frac{691}{5460}\polylog_{y_2}^-+\frac{5}{44}\polylog_{y_4}^--\frac{1}{40}\polylog^-_{y_6}+\frac{1}{24024}\polylog_{y_{14}}^-,\\
            \polylog^-_{y_8}\polylog^-_{y_{10}}&=\frac{43867}{798}\polylog_{y_1}^--\frac{39787}{510}\polylog^-_{y_3}+\frac{77}{3}\polylog_{y_5}^--\frac{11056}{4095}\polylog_{y_7}^-+\frac{5}{66}\polylog_{y_9}^-+\frac{1}{831402}\polylog^-_{y_{19}}.
        \end{align}
    \end{exa}
    The above proposition shows that the difference $M^{(\mathbf{k})}-M^{(\sigma(\mathbf{k}))}$ of permuted Magnus polynomials do not supply all functional equations of non-positive MPLs. In our subsequent paper \cite{KKN}, we shall discuss more about evaluation of the size of $\Ker(\polylog^-_\bullet\circ\pi)$.


\begin{thebibliography}{[EMSS]}
        \bibitem[Bu23]{Bu23}
        V. C. Bui, V. Hoang Ngoc Minh, V. Nguyen Dinh, Q. H. Ngo, \textit{On The Global Renormalization and Regularization of Several Complex Variable Zeta Functions by Computer}, 2023, \url{https://arxiv.org/abs/2209.09849}.
        \bibitem[CM09]{CM09}
        C. Costermans and V. Hoang Ngoc Minh, \textit{non-commutative algebra, multiple harmonic sums and applications in discrete probability}, J. Symbolic Comput. {\bf 44} (2009), no. 7, 801--817.
        \bibitem[Du17a]{D17a}
        G. H. E. Duchamp, V. Hoang Ngoc Minh and Q. H. Ngo, \textit{Harmonic sums and polylogarithms at non-positive multi-indices}, J. Symbolic Comput. {\bf 83} (2017), 166--186.
        \bibitem[Du17b]{D17b}
        G. H. E. Duchamp et al., Mathematical renormalization in quantum electrodynamics via non-commutative generating series, in {\it Applications of computer algebra}, 59--100, Springer Proc. Math. Stat., 198, Springer, Cham, 2017.
        \bibitem[EMS17]{EMS17}
        K. Ebrahimi-Fard, D. Manchon and J. Singer, \textit{The Hopf algebra of ($q$-)multiple polylogarithms with non-positive arguments}, Int. Math. Res. Not. IMRN {\bf 2017}, no. 16, 4882--4922.
        \bibitem[Fu04]{F04}
        H. Furusho, \textit{p-adic multiple zeta values I. p-adic multiple polylogarithms and the p-adic KZ equation}, Inventiones mathematicae \textbf{155.2} (2004): 253--286.
        \bibitem[Go98]{Go98}
        A. Goncharov, \textit{Multiple polylogarithms, cyclotomy and modular complexes}, Math. Res. Lett. \textbf{5}(1998), 497--516.
        \bibitem[GZ17]{GZ17}
        L. Guo and B. Zhang, \textit{Polylogarithms and multiple zeta values from free Rota-Baxter algebras}, Sci. China Math. {\bf 53} (2010), no. 9, 2239--2258.
        \bibitem[IKZ06]{IKZ06}
        K. Ihara, M. Kaneko, D. Zagier, \textit{Derivation and double shuffle relations for multiple zeta values}, Compositio Mathematica, 2006;\textbf{142(2)}:307--338. 
        \bibitem[Ki25]{K25}
        K. Kitamura, \textit{A Combinatorial Study of Multiple Polylogarithms with Non-Positive Multi-Indices (in Japanese)}, master thesis, Osaka University, February 2025.
        \bibitem[KKN]{KKN}
        K. Kitamura, N. Komiyama, H. Nakamura, \textit{in preparation}.
        \bibitem[Ma37]{Ma37}
        W. Magnus, \textit{\"Uber Beziehungen zwitschen h\"oheren Kommutatoren}, J. Reine Angew. Math. \textbf{177}(1937)105--115.
        \bibitem[Na23]{Na23}
        H. Nakamura, \textit{Demi-shuffle duals of Magnus polynomials in a free associative algebra}, Algebr. Comb. {\bf 6} (2023), no. 4, 929--939.
        \bibitem[Wa02]{Wa02}
        M. Waldschmidt, \textit{Multiple Polylogarithms: An Introduction}. In: Agarwal, A.K., Berndt, B.C., Krattenthaler, C.F., Mullen, G.L., Ramachandra, K., Waldschmidt, M. (eds) Number Theory and Discrete Mathematics. Hindustan Book Agency, Gurgaon, 2002. \url{https://doi.org/10.1007/978-93-86279-10-1_1}
    \end{thebibliography}
\end{document}